\documentclass[12pt,leqno]{amsart}

\usepackage{mathrsfs}
\usepackage{amsmath}
\usepackage{amssymb}
\usepackage{amsthm}
\title{Counterexamples to $ C^{\infty} $ well posedness for some
  hyperbolic operators\\ with triple characteristics}
\author{Enrico Bernardi}
\date{}
\address{Dipartimento di Scienze Statistiche, Universit\`a di Bologna,}
\address{Viale Filopanti 5, 40126 Bologna, Italia}
\email{enrico.bernardi@unibo.it}
\author{Tatsuo  Nishitani}
\address{Department of Mathematics, Osaka University,}
\address{ Machikaneyama 1-16, Toyonaka, 560-0043, Osaka, Japan}
\email{nishitani@math.sci.osaka-u.ac.jp}

\chardef\bslash=`\\ 

\DeclareMathOperator{\im}{Im}
\DeclareMathOperator{\re}{Re}

\def\dif{\partial}
\def\R{\mathbb{R}}
\def\C{\mathbb{C}}
\def\lr#1{\langle{#1}\rangle}

\newcommand{\eval}[2][\right]{\relax
  \ifx#1\right\relax \left.\fi#2#1\rvert}

\usepackage{mathtools}
\DeclarePairedDelimiter{\inner}{\langle}{\rangle}
\begin{document}

\begin{abstract}{In this paper we prove that for a class of non-effectively hyperbolic operators with smooth triple characteristics the Cauchy problem is well posed in the Gevrey 2 class, beyond the generic Gevrey class $ 3/2 $ ( see e.g. \cite{Bro}). Moreover we show that this value is optimal.}
\end{abstract}

\maketitle
\renewcommand{\sectionmark}[1]{}
\markboth{Enrico Bernardi and Tatsuo Nishitani}
{Counterexamples to $ C^{\infty} $ well posedness}
\section{Introduction}
\newtheorem{proposition}{Proposition}[section]
\newtheorem{lemma}{Lemma}[section]
\newtheorem{theorem}{Theorem}[section]

Hyperbolic operators with double characteristics have been thoroughly
investigated in the past years, and at least in the case when there is no
transition between different types on the set where the principal symbol
vanishes of order $ 2 $, essentially everything is known, see e.g.
\cite{NBook} and \cite{BNJ} for a general survey and \cite{IP} and \cite{Ho1} for
classical introductions. The
algebraic classification of the spectrum of the fundamental matrix of the principal
symbol evaluated at a double point allows us to deduce the behavior
of the operator in the $ C^{\infty} $ and Gevrey categories as far as
the well posedness of the Cauchy problem is concerned. In particular,
when real eigenvalues exist, the so called effectively hyperbolic
case, then we have well posedness regardless of the lower order terms.

These  spectral invariants are not available in general when studying operators with
symbols vanishing of order greater or equal to $ 3 $, therefore
much less is known in this case. There is one object, though, that allows some
classification even in these cases, namely the propagation cone of the
principal symbol, i.e. the symplectic dual of the hyperbolicity cone.

More precisely we recall that, denoting by $ p_{z} $ the localization
of the principal symbol $ p $ of $ P(x,D) $ at a multiple point $ z $,
the propagation cone $ C_{z} $ is defined by
\[
C_{z} = \{X \in T_{z}(T^{*}\R^{n+1})| \sigma(X,Y) \leq 0, ~~ \forall Y
\in \Gamma_{z}\},
\]

where the hyperbolicity cone $  \Gamma_{z} $ is defined as the
connected component of $ N = (0;1,\ldots,0) $ of the set $ \{X\in
T_{z}(T^{*}\R^{n+1})|p_{z}(X) \neq 0\} $, assuming that $ p(x,\xi) $
is hyperbolic with respect to $ \xi_{0} $.

When $ C_{z} $  happens to be transversal to the tangent plane to the
manifold of multiple points, we are again effectively hyperbolic as it
were,
i.e. if characteristics are double, it can be shown that this is
equivalent to  the spectrum of the fundamental
matrix  containing real eigenvalues (\cite{N94}, \cite{Ho1}). When
this happens in a higher order multiplicity situation and the lower
order terms satisfy a generic Ivrii-Petkov vanishing condition, it is
known that we have well posedness in $ C^{\infty} $. 
See \cite{KNW} for a very complete analysis of this situation and
\cite{BBPK} for some new recent  results in triple characteristics
of an effectively hyperbolic type.
One strongly suspects that when this transversality condition fails, it may be always
possible to choose some suitable lower order terms satisfying Ivrii-Petkov
conditions and still end up with an ill posed problem in $  C^{\infty}
$.

At least in the case of triple characteristics this behavior has been proved in
a number of papers, see e.g. \cite{BGE}, \cite{Iwa}, \cite{N94}, but
the principal symbol had to satisfy some strong
factorization conditions, where one or all of the roots had to be $ C^{\infty} $. 

In this paper we  prove a well posedness result in the Gevrey
category for a simple model hyperbolic operator with triple characteristics,
when however there are no regular roots, i.e. the principal symbol
cannot be smoothly factorized, and moreover whose propagation cone is not
transversal to the triple manifold, thus confirming that conjecture,
albeit for a limited class of operators. On the other hand here we are
able not only to disprove $ C^{\infty} $ well posedness, but we can
actually estimate the precise Gevrey threshold where well posedness
will cease to hold, by exhibiting a special class of solutions, through
which we can violate weak necessary solvability conditions.
This threshold will appear at $ s=2 $, thus beyond the canonical value
of $ s =\frac{3}{2} $ dictated by the classical result of Bronshtein,
\cite{Bro}.
The choice of the lower order terms will be the easiest possible,
i.e. zero. It is thus all the more surprising that a very regular
operator, with analytic (polynomial) coefficients, and reduced just to its
principal symbol should have this bad behavior, with respect to $
C^{\infty} $ well posedness.

We consider the operator

\begin{equation}
  \label{eq:10}
  P(x,D) = D_{0}^{3} - \Omega(x,D') D_{0} + B(x,D')
\end{equation}

with 
$$ \Omega(x,D') = D_{1}^{2} + x_{1}^{2}D_{n}^{2} 
$$ and $ B(x,D') =
-b_{0}x_{1}^{3}D_{n}^{3}$.

Here $ x = (x_0,x_1, \ldots, x_n) = (x_{0},x')=
(x_{0},x_{1},x'',x_{n}) \in \R^{n+1} $ and the local
estimates below will be proven in a neighborhood of $ x=0 $.
Clearly hyperbolicity is equivalent to $ b_{0}^{2} \leq \frac{4}{27}
$. We will also assume that the principal symbol vanishes exactly of
order $ 3 $ on the triple manifold $ \Sigma_{3} $, thus we will require  $ |b_{0}| <
\frac{2}{3\sqrt{3}} $, i.e. outside  $ \Sigma_{3} $ $ P $ is strictly hyperbolic.

Let us recall that we say that $f(x)\in C^{\infty}({\mathbb R}^n)$ belongs to $\gamma^{(s)}({\mathbb R}^n)$, the Gevrey space of order $s$, where $ s \geq 1 $,  if for any compact set $K\subset {\mathbb R}^n$ there exist $C>0$, $h>0$ such that
\begin{equation}
\label{Gevreyineq}
\bigl|\dif^{\alpha}_xf(x)\bigr|\leq Ch^{|\alpha|}|\alpha|!^s,\quad x\in K
\end{equation}
for every $\alpha\in {\mathbb N}^n$. In particular
$\gamma^{(1)}({\mathbb R}^n)$ is the space of real analytic functions
on ${\mathbb R}^n$.

We also recall  that the Cauchy problem for $P$ is said to be  locally solvable in $\gamma^{(s)}$ at the origin if for any $\Phi=(u_0,u_1)\in (\gamma^{(s)}({\mathbb R}^n))^2$, there exists a neighborhood $U_{\Phi}$ of the origin such that the Cauchy problem
\[
\left\{\begin{array}{ll}
Pu=0\quad \mbox{in~} U_{\Phi}\\
D_0^ju(0,x')=u_j(x'),\quad j=0,1,\quad x\in U_{\Phi}\cap\{x_0=0\}
\end{array}\right.
\]
has a solution $u(x)\in C^{\infty}(U_{\Phi})$.

The main results in this paper are then precisely stated:

\begin{theorem}
\label{thm:main:1}
Assume that $0 < |b_{0}| <  \frac{2}{3\sqrt{3}}$. Then the Cauchy problem for $P$ is well posed in the Gevrey $2$ class.
\end{theorem}

That this is actually the best one can hope for is proven in
\begin{theorem}
\label{thm:main:2}
If $s>2$, it is possible to choose $ b_{0} \in ]0,\frac{2}{3\sqrt{3}}[
$ such that  the Cauchy problem for $P$ is not locally solvable at the origin in the Gevrey $s$ class.
\end{theorem}

The paper is organized as follows: in Section
\ref{sec:estim-gevr-class} we prove a simple, classical energy
estimate for our model operator in the Gevrey $ s $ category, with $ s
\leq 2$ which proves Thoerem \ref{thm:main:1}.

In Section \ref{sec:optim-gevr-index} we recall a number of results
from \cite{Si}, \cite{NZ} and \cite{BNJ} on the entire functions
related to a Stokes phenomenon for an important ODE associated
with the necessary conditions. We finally prove Theorem
\ref{thm:main:2} via a standard functional analytic argument, involving exponential
estimates.

Eventually in Section \ref{sec:factorization} we verify that the
geometrical conditions on the propagation cone and the regularity of
the roots for the principal symbol of our model hold true.

\section{Estimates in Gevrey classes}
\label{sec:estim-gevr-class}

We define $\displaystyle \inner{u,v} =
\int_{\R^{n}}\hat{u}(x_{0},x_{1},x',\xi_{n})\overline{\hat{v}}(x_{0},x_{1},x',\xi_{n})dx_{1}dx'
$ with $ \hat{u} $ denoting the partial Fourier transform with respect
to $ x_{n} $. In a similar way we have for the $ L^{2} $ norm $ \displaystyle \|u\|^{2} = \int_{\R^{n}}|\hat{u}(x_{0},x_{1},x',\xi_{n})|^{2}dx_{1}dx' $.

Since we are dealing a rather simple and straightforward model
operator we are not going to deploy the techniques used e.g. in \cite{BNJ} of
Weyl-Gevrey calculus of pseudo-differential operators.We could
certainly  apply them here, but at the price of rendering the computations very
heavy, and then  for a very limited advantage in generality.

Therefore the symbol $ W(x_0) = \exp(2\tau\lr{\xi_{n}}^{\frac{1}{s}}(x_{0}-a)) $ with some $a>0$ defined below
will function as a  Gevrey weight in a naive, still correct, way. 

The harmonic oscillator $ \Omega $ is defined as $ \Omega =
D_{1}^{2} + x_{1}^{2}\xi_{n}^{2} $.

 Before dealing
with the operator (\ref{eq:10}) itself, we need a preliminary result
on the multiplier operator $ M $. Let
\[
E_j(u(x_0))=\|D_0^ju(x_0)\|^2+\|D_1^ju(x_0)\|^2+\|(x_1\xi_n)^ju(x_0)\|^2.
\]

Assuming that $ \theta > 0 $ we start by proving the following 

\begin{lemma}
\label{lemma:2.1}
  Let $ M = D_{0}^{2} - \theta\Omega $. Then there esists  $C > 0 $ such
  that for any $  s \geq 1, s\in \R $ and any $  \tau  $ large
  enough we have for any $u\in C_0^{\infty}(\R^n)$
  \begin{equation}
    \label{eq:22}
    \begin{split}
    \int_{0}^{\infty}W\|Mu\|^{2}dx_{0} \geq
   CW(0)\big(\tau\lr{\xi_{n}}^{\frac{1}{s}}E_1(u(0))
   +\tau^3\lr{\xi_{n}}^{\frac{3}{s}}E_0(u(0))\big)\\
   + C\tau^{2}\int_{0}^{\infty}W\lr{\xi_{n}}^{\frac{2}{s}}E_1u(x_0))dx_{0}
   +C\tau^4\int_0^{\infty}W\lr{\xi_{n}}^{\frac{4}{s}}E_0(u(x_0))dx_0,
   \end{split}
  \end{equation}
where $ W = \exp(2\tau\lr{\xi_{n}}^{\frac{1}{s}}(x_{0}-a)) $ and $\lr{\xi_n}=\sqrt{1+\xi_n^2}$.
\end{lemma}

\begin{proof} We compute

\begin{align}
  \label{eq:23}
  2i\im\inner{Mu,D_{0}u} &= 2i\im\inner{D_{0}^{2}u,D_{0}u}
  -\theta2i\im\inner{\Omega u, D_{0}u}\notag\\
&= D_{0}\{\|D_{0}u\|^{2} + \theta\inner{\Omega u,u}\} \notag\\
&= D_{0}E(u(x_{0})).\notag\\
\end{align}

Therefore we have
\begin{equation}
  \label{eq:24}
  \begin{split}
  2\int_{0}^{\infty}W\im\inner{Mu,D_{0}u} dx_{0} = W(0)E(u(0))\\
   +
    2\tau\int_{0}^{\infty}W E(u(x_{0}))\lr{\xi_{n}}^{\frac{1}{s}}dx_{0}.
    \end{split}
\end{equation}

Since $ \lr{\Omega u,u} = \|D_1u\|^2+\|x_1\xi_n u\|^2 $
we have

\begin{equation}
  \label{eq:25}
  \begin{split}
   2\int_{0}^{\infty}W\im\inner{Mu,D_{0}u} dx_{0} = W(0) E(u(0))\\
       +
2\tau\int_{0}^{\infty}W\lr{\xi_{n}}^{\frac{1}{s}}\{\|D_{0}u\|^{2} +
   \theta\|D_1u\|^{2} + \theta\xi_{n}\|x_1\xi_n u\|^{2}dx_{0} .
   \end{split}
\end{equation}

Using Cauchy-Schwarz inequality  we see 
\begin{equation}
\label{eq:add_1}
\begin{split}
\int_0^{\infty}W\|Mu\|^2\geq \tau \lr{\xi_{n}}^{\frac{1}{s}}W(0)E(u(0))\\
+\tau^2\int_0^{\infty}W\lr{\xi_{n}}^{\frac{2}{s}}\big(\|D_0u\|^2
+\theta\|D_1u\|^2+\theta \|x_1\xi_nu\|^2\big)dx_0.
\end{split}
\end{equation}
Repeating similar arguments we have
\[
\int_0^{\infty}W\|D_0u\|^2dx_0\geq \tau\lr{\xi_{n}}^{\frac{1}{s}}W(0)\|u(0)\|^2+\tau^2\int_0^{\infty}W\lr{\xi_{n}}^{\frac{2}{s}}\|u\|^2dx_0
\]
and replacing ($\int_0^{\infty}W\|D_0u\|^2dx_0)/2$ by the above estimate the right-hand side of \eqref{eq:add_1} is bounded from below by
\begin{align*}
W(0)\big(\tau\lr{\xi_{n}}^{\frac{1}{s}}E_1(u(0))+\frac{1}{2}\tau^2\lr{\xi_{n}}^{\frac{3}{s}}E_0(u(0))\big)\\
+\tau^2\int_0^{\infty}W\lr{\xi_{n}}^{\frac{2}{s}}\big(E_1(u(x_0))dx_0+
+\tau^2\lr{\xi_{n}}^{\frac{2}{s}}E_0(u(x_0))\big)dx_0.
\end{align*}
It is easy to see that (\ref{eq:22}) holds.
\end{proof}

We now move to the proof of Theorem \ref{thm:main:1}.

\begin{proof} First notice that if $ b_{0}=0$  the result is a trivial
  consequence of the double characteristics theory, and in that case
  we do have $ C^{\infty} $ well posedness, as it will also become
  clear from the estimates below.  That is why we will assume that $ b_{0} \neq
  0$. We will make use of standard energy estimates.


We choose $ \theta = \frac{1}{3} $ and with $\displaystyle M(x,D) = D_{0}^{2} - \frac{\Omega}{3} $  compute

\begin{equation}
  \label{eq:11}
  \begin{split}
  2i\im\langle Pu, Mu \rangle = 2i\im\langle \Bigl(D_{0}M -
  \frac{2}{3}\Omega D_0 - B\Bigr)u   , Mu \rangle \\
 = D_{0}\Bigl\{\|Mu\|^{2}\Bigr\} + 2i\im\langle - \frac{2}{3}\Omega
D_{0}u   , D_{0}^{2}u \rangle + 2i\im\langle  \frac{2}{3}\Omega
D_{0}u   , \frac{\Omega}{3}u  \rangle  \\
\phantom{=} +2i\im\langle -b_{0}x_{1}^{3}\xi_{n}^{3}u   ,D_{0}^{2}u
\rangle + 2i\im\langle  -b_{0}x_{1}^{3}\xi_{n}^{3}u  , -\frac{\Omega}{3}u  \rangle.
\end{split}
\end{equation}

From (\ref{eq:11}) we get
\begin{equation}
  \label{eq:12}
  2i\im\langle Pu, Mu \rangle  = D_{0}\mathcal{E}(u) + \mathcal{R}(u),
\end{equation}

where $ \mathcal{R}(u) = \frac{b_{0}}{3}\langle
[D_{1}^{2},x_{1}^{3}]\xi_{n}^{3}u,u    \rangle $ and
\begin{equation}
  \label{eq:13}
  \mathcal{E}(u) = \|Mu\|^{2} + \frac{2}{3}\langle \Omega D_{0}u   ,D_{0}u
  \rangle + \frac{2}{9}\|\Omega u\|^{2} + 2b_{0}\re\langle
  x_{1}^{3}\xi_{n}^{3}u   ,D_{0}u \rangle .
\end{equation}

From (\ref{eq:13}) we have
\begin{equation}
  \label{eq:14}
  \begin{split}
  \mathcal{E}(u) &= \|Mu\|^{2} +  2b_{0}\re\langle x_{1}^{2}\xi_{n}^{2}u
  ,x_{1}\xi_{n}D_{0}u \rangle \\
& \phantom{=} + \frac{2}{3}\Bigl(\|D_{1}D_{0}u\|^{2} +
\|x_{1}\xi_{n}D_{0}u\|^{2}\Bigr) \\
& \phantom{=} + \frac{2}{9}\Bigl(\|D_{1}^{2}u\|^{2} +
\|x_{1}^{2}\xi_{n}^{2}u\|^{2} + 2\re\langle D_{1}^{2}u
,x_{1}^{2}\xi_{n}^{2}u \rangle\Bigr) .
\end{split}
\end{equation}

We write (\ref{eq:14}) like this:
\begin{equation}
  \label{eq:15}
  \begin{split}
  \mathcal{E}(u) &= \|Mu\|^{2}  +
\frac{2}{3}\|D_{1}D_{0}u\|^{2} \\
& \phantom{=} +  \Bigl\| \sqrt{\frac{2}{3}}x_{1}\xi_{n}D_{0}u +
b_{0}\sqrt{\frac{3}{2}}x_{1}^{2}\xi_{n}^{2}u\Bigr\|^{2} + \frac{2}{9}\|D_{1}^{2}u\|^{2} \\
& \phantom{=} + \frac{2}{9}\Bigl(1  -
\frac{27}{4}b_{0}^{2}\Bigr)\|x_{1}^{2}\xi_{n}^{2}u\|^{2} +
\frac{4}{9}\re\langle D_{1}^{2}u   ,x_{1}^{2}\xi_{n}^{2}u \rangle .
\end{split}
\end{equation}

Noticing that $ \re\inner{D_{1}^{2}u,x_{1}^{2}u} = \|x_{1}D_{1}u\|^{2}
- \|u\|^{2}$ we get from (\ref{eq:15}) that

\begin{equation}
  \label{eq:151}
  \begin{split}
  \mathcal{E}(u) &= \|Mu\|^{2}  +
\frac{2}{3}\|D_{1}D_{0}u\|^{2} \\
& \phantom{=} +  \Bigl\| \sqrt{\frac{2}{3}}x_{1}\xi_{n}D_{0}u +
b_{0}\sqrt{\frac{3}{2}}x_{1}^{2}\xi_{n}^{2}u\Bigr\|^{2} + \frac{2}{9}\|D_{1}^{2}u\|^{2} \\
& \phantom{=} + \frac{2}{9}\Bigl(1  -
\frac{27}{4}b_{0}^{2}\Bigr)\|x_{1}^{2}\xi_{n}^{2}u\|^{2} + \frac{4}{9}
\|x_{1}D_{1}u\|^{2} - \frac{4}{9}\xi_{n}^{2}\|u\|^{2}.
\end{split}
\end{equation}

Multiplying by $ W $ and integrating from $ 0 $ to $ \infty $ we have

\begin{equation}
  \label{eq:152}
  \begin{split}
 &\int_{0}^{\infty}2W\im\inner{Pu,Mu}dx_{0} \\
 &= W(0)\mathcal{E}(u)(0)+ 2\tau\lr{\xi_{n}}^{\frac{1}{s}}\int_{0}^{\infty}W\Bigl\{ \|Mu\|^{2}  +
\frac{2}{3}\|D_{1}D_{0}u\|^{2} \\
& \phantom{=} +  \Bigl\| \sqrt{\frac{2}{3}}x_{1}\xi_{n}D_{0}u +
b_{0}\sqrt{\frac{3}{2}}x_{1}^{2}\xi_{n}^{2}u\Bigr\|^{2} + \frac{2}{9}\|D_{1}^{2}u\|^{2} \\
& \phantom{=} + \frac{2}{9}\Bigl(1
-\frac{27}{4}b_{0}^{2}\Bigr)\|x_{1}^{2}\xi_{n}^{2}u\|^{2} + 
\frac{4}{9}\|x_{1}D_{1}u\|^{2} - \frac{4}{9}\xi_{n}^{2}\|u\|^{2} \Bigr\} dx_{0}\\
& \phantom{=} -2b_{0}\xi_{n}^{3}\int_{0}^{\infty}W 
\re\inner{x_{1}^{2}u,D_{1}u}dx_{0} .
\end{split}
\end{equation}

Recalling (\ref{eq:22}) from Lemma (\ref{lemma:2.1}) now with $ 1 \leq s \leq 2 $ we can dispose of
the negative contribution in (\ref{eq:152}) $ - \frac{4}{9}\xi_{n}^{2}\|u\|^{2}  $.

Let us now deal with the remainder term 
$$ 
\mathcal{R}(u) =  -2b_{0}\xi_{n}^{3}\int_{0}^{\infty}W
\re\inner{x_{1}^{2}u,D_{1}u}dx_{0}. 
$$

Applying Cauchy-Schwarz inequality twice we get
\begin{equation}
  \label{eq:26}
  \begin{split}
  |2\re\inner{x_{1}^{2}\xi_{n}^{2}u,\xi_{n}^{1 }\lr{\xi_n}^{- \frac{1}{s}}D_{1}u}|
  &\leq \|x_{1}^{2}\xi_{n}^{2}u\|^{2} + \lr{\xi_n}^{2
    -\frac{2}{s}}\|D_{1}u\|^{2} \\
& \phantom{\leq} = \|x_{1}^{2}\xi_{n}^{2}u\|^{2} + \lr{\xi_n}^{2
    -\frac{2}{s}}\inner{D_{1}^{2}u,u} \\
& \phantom{\leq} \leq \|x_{1}^{2}\xi_{n}^{2}u\|^{2} + \lr{\xi_n}^{4
    -\frac{4}{s}}\|u\|^{2} + \|D_{1}^{2}u\|^{2}.
    \end{split}
\end{equation}
It is clear that $\|D_0^2u\|\leq 4(\|Mu\|^2+\|x_1^2\xi_n^2 u\|^2+\|D_1^2u\|^2)$. 
Using (\ref{eq:152}) and $ 1 + \frac{2}{s} \geq 2 \geq  4 - \frac{4}{s} $ we obtain for any $u\in C_0^{\infty}(\R^n)$
\begin{equation}
  \label{eq:27}
  \begin{split}
  \int_{0}^{\infty}W\|Pu\|^{2}dx_{0} \geq CW(0)\sum_{j=0}^2\tau^{4-3j/2}\lr{\xi_{n}}^{\frac{5-2j}{s}}E_j(u(0))\\
  +  C\sum_{j=0}^2\tau^{6-2j}\int_0^{\infty} W\lr{\xi_{n}}^{\frac{6-2j}{s}}E_j(u(x_0))dx_0  
  \end{split}
\end{equation}

if $ \tau $ is large enough and $ 1 \leq s \leq 2 $. If we choose
$s=3/2$ so that we have
$\lr{\xi_n}^{4}E_0(u(x_0))=E_0(\lr{\xi_n}^2u(x_0))$ which control any
lower order term and we arrive at the Bronshtein's theorem (see \cite{Bro}). 

 Let $s=2$ and 
\[
{\tilde E}_j(u(x_0))= \|D_0^ju(x_0)\|^2+\|D_1^ju(x_0)\|^2+\|(x_1D_n)^ju(x_0)\|^2.
\]

  Then for any $u\in  C_0^{\infty}(\R^{n+1})$ vanishing in $x_0\geq a$ we integrate \eqref{eq:22} with respect to $\xi_n$  we get
\begin{equation}
    \label{eq:22bis}
    \begin{split}
    \int_{0}^{a}\|e^{\tau\lr{D_n}^{\frac{1}{2}}(x_0-a)}Pu\|^2dx_{0} 
   \geq C\sum_{j=0}^2 {\tilde E}_j\big(e^{-\tau a \lr{D_n}^{\frac{1}{2}}}\lr{D_n}^{\frac{5-2j}{4}}u(0)\big)\\
   +  C\int_{0}^{a}\sum_{j=0}^2{\tilde E}_j\big(e^{\tau\lr{D_n}^{\frac{1}{2}}(x_0-a)}\lr{D_n}^{\frac{6-2j}{4}}u(x_0)\big).
  \end{split}
  \end{equation}

Let us denote $\lr{\xi}=\sqrt{1+\sum_{j=1}^n\xi_j^2}$. Note that
\[
\lr{D}^sx_1^k=\sum_{\ell=0}^k\frac{1}{\ell!}x_1^{k-\ell}\phi_{s\ell}(D),\;\;\phi_{s\ell}(\xi)=(-i)^{\ell}\frac{\dif^{\ell}}{\dif\xi_1^{\ell}}\lr{\xi}^s.
\]
Then writing $\lr{D}^sPu=(P+R)\lr{D}^su$ it is easy to check
\begin{lemma}
\label{lemma:2.2}
 For any $s\in\R$ there exist $ C=C_s > 0 $, $\tau=\tau_s>0$ such
  that for  any $u\in C_0^{\infty}(\R^{n+1})$ vanishing in $x_0\geq a$
  \begin{equation}
    \label{eq:22ter}
    \begin{split}
    \int_{0}^{a}\|e^{\tau\lr{D_n}^{\frac{1}{2}}(x_0-a)}\lr{D}^sPu\|^2dx_{0} \\
    \geq C\sum_{j=0}^2 {\tilde E}_j\big(e^{-\tau a \lr{D_n}^{\frac{1}{2}}}\lr{D_n}^{\frac{5-2j}{4}}\lr{D}^su(0)\big)\\
   +  C\int_{0}^{a}\sum_{j=0}^2{\tilde E}_j\big(e^{\tau\lr{D_n}^{\frac{1}{2}}(x_0-a)}\lr{D_n}^{\frac{6-2j}{4}}\lr{D}^su(x_0)\big)dx_0.
     \end{split}
  \end{equation}
\end{lemma}

Let $E=\{Pv\mid v\in C_0^{\infty}(\R^{n+1}\cap\{x_0<a\})\}$ and let $s>0$ large. Consider the anti-linear functional
\[
\Phi:Pv\mapsto \sum_{j=0}^2(\phi_{2-j},D_0^jv(0))+\int_0^{a}(f,v)dx_0
\]
where we assume that 
\begin{equation}
\label{eq:assump}
\begin{split}
e^{\tau a\lr{D_n}^{\frac{1}{2}}}\lr{D_n}^{\frac{-(5-2j)}{4}}\lr{D}^s\phi_{2-j}\in L^2(\R^n),\\
e^{-\tau\lr{D_n}^{\frac{1}{2}}(x_0-a)}\lr{D_n}^{-\frac{3}{2}}\lr{D}^sf\in L^2((0,a)\times\R^n).
\end{split}
\end{equation}
 From Lemma \ref{lemma:2.2} we have
\begin{align*}
\sum_{j=0}^2\big|(\phi_{2-j},D_0^jv(0))\big|+\Big|\int_0^{a}(f,v)dx_0\Big|\\
\leq \big(\sum_{j=0}^2\|e^{\tau a\lr{D_n}^{\frac{1}{2}}}\lr{D_n}^{\frac{-(5-2j)}{4}}\lr{D}^s\phi_{2-j}\|^2\big)^{1/2}\\
\times \big(\sum_{j=0}^2\|e^{-\tau a\lr{D_n}^{\frac{1}{2}}}\lr{D_n}^{\frac{5-2j}{4}}\lr{D}^{-s}D_0^jv(0)\|^2\big)^{1/2}\\
+\Big(\int_0^{a}\|e^{-\tau\lr{D_n}^{\frac{1}{2}}(x_0-a)}\lr{D_n}^{-\frac{3}{2}}\lr{D}^sf\|^2dx_0\Big)^{1/2}\\
\times \Big(\int_0^{a}\|e^{\tau\lr{D_n}^{\frac{1}{2}}(x_0-a)}\lr{D_n}^{\frac{3}{2}}\lr{D}^{-s}v\|^2dx_0\Big)^{1/2}\\
\leq C\big(\int_0^{a}\|e^{\tau\lr{D_n}^{\frac{1}{2}}(x_0-a)}\lr{D}^{-s}Pv\|^2dx_0\Big)^{1/2}.
\end{align*} 
From the Hahn-Banach theorem $\Phi$ can be extended to a bounded linear functional on $\{u\mid e^{\tau \lr{D_n}^{1/2}(x_0-a)}\lr{D}^{-s}u\in L^2((0,a)\times \R^n)\}$. Then there exists $u$ such that 
\[
\int_0^{a}\|e^{-\tau\lr{D_n}^{\frac{1}{2}}(x_0-a)}\lr{D}^su\|^2dx_0<+\infty
\]
and satisfies
\[
T(g)=\int_0^{a}(u,g)dx_0.
\]
When $g=Pv$ one has
\[
\sum_{j=0}^2(\phi_j,D_0^jv(0))+\int_0^{a}(f,v)dx_0=\int_0^{a}(u,Pv)dx_0 
\]
fro any $v\in C_0^{\infty}(\R^{n+1})$ with $v=0$ for $x_0\geq a$. Choosing $v$ so that $v\in C_0^{\infty}(\{0<x_0<a\})$ we see that $Pu=f$ in $(0,a)\times \R^n$ since $P^*=P$. Thus we have
\begin{align*}
\sum_{j=0}^2(\phi_{2-j},D_0^jv(0))=-i(u(0),D_0^2v(0))\\
-i(D_0u(0),D_0v(0))-i(D_0^2u(0),v(0))+i(\Omega u(0),v(0)).
\end{align*}
From this it follows that
\[
u(0)=i\phi_0,\;D_0u(0)=i\phi_1,\;D_0^2u(0)=i\phi_2+\Omega\phi_0.
\]
Since $e^{-\tau \lr{D_n}^{\frac{1}{2}}(x_0-a)}\lr{D}^{s}u=U\in L^2((0,a)\times \R^n)$ we have
\[
u=e^{\tau\lr{D_n}^{\frac{1}{2}}(x_0-a)}\lr{D}^{-s}U
\]
hence it is cclear that $u\in L^2((0,a); H^s(\R^n))$. Since we have
$Pu\in L^2((0,a);H^{s-3/2}(\R^n))$ from the assumption then from
Theorem B.2.9 (\cite{Ho2}, vol.3) it follows that
\[
D_0^ju\in L^2((0,a);H^{s-3/2-j}(\R^n))
\]
for $j=0,1,2,3$. Thus we get a smooth solution in $(0,a)\times \R^n$ provided \eqref{eq:assump} is verified and choosing $s$ large. 

\end{proof}

\section{Optimality of the Gevrey index}
\label{sec:optim-gevr-index}

\subsection{Sibuya's results}
\label{sec:sibuyas-results}

The differential equation 
\begin{equation}
  \label{eq:5}
  w''(y) = (y^{3} + \zeta y)w(y)
\end{equation}

will play a very important role in the construction of the family of
solutions leading to the optimality of the Gevrey index $ s=2 $. 

Therefore we  recap briefly, in this special 
setting, the general theory of subdominant 
solutions and Stokes coefficients for the equation (\ref{eq:5}), following  the 
presentation found, for example, in the book of Sibuya \cite{Si}. 

Theorem 6.1 in \cite{Si} states that the differential equation 
(\ref{eq:5}) has a \textit{unique} solution 
\[
w(y;\zeta)={\mathcal Y}(y;\zeta)
\]
such that 
\begin{itemize}
\item[\rm(i)] ${\mathcal Y}(y;\zeta)$ is an entire function 
of $(y,\zeta)$.
\item[\rm(ii)] ${\mathcal Y}(y;\zeta)$ and its derivative ${\mathcal Y'}(y;\zeta)$  admit an asymptotic 
representation
\begin{equation}
\label{as1}
{\mathcal Y}(y;\zeta)\sim y^{-3/4}\Bigl[1+\sum_{N=1}^{\infty}
B_Ny^{-N/2}\Bigr]\exp{[-E(y;\zeta)]}
\end{equation}
\begin{equation}
\label{as2}
{\mathcal Y'}(y;\zeta)\sim y^{3/4}\Bigl[-1+\sum_{N=1}^{\infty}
C_Ny^{-N/2}\Bigr]\exp{[-E(y;\zeta)]}
\end{equation}
uniformly on each compact set in the $\zeta$ space as $y$ goes to infinity 
in any closed subsector of the open sector
\[
|\arg y|<\frac{3\pi}{5};
\]
moreover 
\[
E(y;\zeta)=\frac{2}{5}y^{5/2}+\zeta y^{1/2}
\]
 and $B_N$, $C_N$ are polynomials in $\zeta$.
 \end{itemize}

 \noindent
 We note that if we set 
 $\displaystyle{\omega=\exp{[i\frac{2\pi}{5}]}}$ and
 \[
 {\mathcal Y}_k(y;\zeta)={\mathcal Y}(\omega^{-k}y;\omega^{-2k}\zeta)
 \]
 where $k=0,1,2,3,4$ then all the five functions ${\mathcal Y}_k(y;
 \zeta)$ solve (\ref{eq:5}). In particular 
 ${\mathcal Y}_0(y;\zeta)={\mathcal Y}(y;\zeta)$.
 Let us denote
 \[
 Y=y^{-3/4}\Bigl[1+\sum_{N=1}^{\infty}
B_Ny^{-N/2}\Bigr]\exp{[-E(y;\zeta)]}
\]
then we have
\begin{itemize}
\item[\rm(i)]
${\mathcal Y}_k(y;\zeta)$ is an entire function of 
$(y,\zeta)$.
\item[\rm(ii)] ${\mathcal Y}_k(y;\zeta)
\sim Y(\omega^{-k}y;\omega^{-2k}\zeta)$
uniformly on each compact set in the $\zeta$ space as 
$y$ goes to infinity  in any closed subsector of the open sector 
\[
|\arg y-\frac{2k}{5}\pi|<\frac{3\pi}{5}.
\]
\end{itemize}
Let $S_k$ denote the open sector defined by $\displaystyle{|\arg y-\frac{2k}{5}\pi|<\frac{\pi}{5}}$.
We say that a solution of (\ref{eq:5}) is subdominant 
in the sector $S_k$ if it tends to $0$ as $y$ tends to infinity along 
any direction in the sector $S_k$. Analogously a solution is called 
dominant in the sector $S_k$ if this solution tends to $\infty$ as 
$y$ tends to infinity along any direction in the sector $S_k$. 

Since 
\begin{equation}
\label{eqn:11:0:6}
{\mathsf {Re}}[y^{5/2}]>0\quad {\rm for}\;\;y\in S_0
\end{equation}
and ${\mathsf {Re}}[y^{5/2}]<0$ for $y\in S_{-1}=S_4$ and for 
$S_1$ the solution ${\mathcal Y}_0(y;\zeta)$ is subdominant 
in $S_0$ and dominant in $S_4$ and $S_1$. Similarly ${\mathcal Y}_k
(y;\zeta)$ is subdominant in $S_k$ and dominant in $S_{k-1}$ 
and $S_{k+1}$. It is clear that ${\mathcal 
Y}_{k+1}$ and ${\mathcal Y}_{k+2}$ are linearly independent. 
Therefore ${\mathcal Y}_k$ is a linear combination of those two:
\[
{\mathcal Y}_k(y;\zeta)=C_k(\zeta){\mathcal Y}_{
k+1}(y;\zeta)+{\tilde C}_k(\zeta){\mathcal Y}_{
k+2}(y;\zeta).
\]
The above relation, connection formula for ${\mathcal Y}_k(y;\zeta)$ 
and the coefficients $C_k$, ${\tilde C}_k$ are called the Stokes 
coefficients for ${\mathcal Y}_k(y;\zeta)$. We summarize
 in the following statement some of the known and useful facts about 
 the Stokes coefficients for our particular equation (\ref{eq:5}).
 Proofs can be found in Chapter 5 of \cite{Si}.

 \begin{proposition}
 \label{pro:11:0:a}
 The following results hold.
 \begin{itemize}
 \item[\rm(i)] ${\tilde C}_k(\zeta)=-\omega$, $\forall k,~$ and $\zeta$,
 \item[\rm(ii)] $C_k(\zeta)
 =C_0(\omega^{-2k}\zeta)$, 
 $\forall k,\;\zeta$ and 
 $C_0(\zeta)$ is an entire 
 function of $\zeta$,
 \item[\rm(iii)] For each fixed $\zeta$ there exists $k$ 
 such that $C_k(\zeta)\neq 0$,
 \item[\rm(iv)] $C_k(0)=1+\omega$,\quad $\forall k$,
 \item[\rm(v)] $\partial_{\zeta}C_0(\zeta)\vert_{\zeta=0}\neq 0$.
  \end{itemize}
 \end{proposition}
 We also have
 \begin{proposition}
 \label{pro:11:0:b} If we set
 \[
 S_k(\zeta)=\left[\begin{array}{cc}C_k(\zeta)&1\\
 -\omega&0\end{array}\right],\quad k=0,1,2,3,4
 \]
 then we have
 \begin{equation}
 \label{eqn:11:0:7}
 S_4(\zeta)\cdot S_3(\zeta)\cdot S_2(
 \zeta)\cdot S_1(\zeta)\cdot S_0(\zeta)=
 \left[\begin{array}{cc}1&0\\
 0&1\end{array}\right].
 \end{equation}
 \end{proposition}
 The proof of Proposition \ref{pro:11:0:b} is straightforward. 
 Applying this proposition we have an interesting result.
 
 \begin{proposition}
\label{pro:11:0:c}
 \eqref{eqn:11:0:7} is 
 equivalent to 
 \[
 C_k(\zeta)+\omega^2C_{k+2}(\zeta)C_{k+3}(\zeta)-\omega^3=0\quad {\rm mod}~ 5.
 \]
 Or otherwise stated
 \[
 C_{0}(\zeta)+\omega^2C_{0}(\omega \zeta)C_{0}(\omega^4\zeta)-\omega^3=0,\quad \forall \zeta\in{\mathbb C}.
 \]
 \end{proposition}
 \noindent
 Proof: A straightforward computation from (\ref{eqn:11:0:7}). 
 
 \medskip
 
 We now state a key lemma which is proved in \cite{BNJ}. We repeat here the short proof.
 
 \begin{lemma}
 \label{lem:11:0:a}
 The Stokes coefficient $ C_0(\zeta)$ vanishes in at least one {\rm (non zero)} $\zeta_0$.
 \end{lemma}

 \noindent
 Proof: Suppose that $C_{0}(\zeta)\neq 0$ for all $\zeta\in {\mathbb C}$. Then from Proposition \ref{pro:11:0:c} that it follows that 
 $C_{0}(\zeta)\neq \omega^3$ for all $\zeta\in {\mathbb C}$. Since
 $C_{0}(\zeta)$ is an entire function  Picard's Little Theorem implies
 that $C(\zeta)$ would be constant because $C_{0}(\zeta)$ avoids two
 distinct values $0$ and $\omega^3$.  But this contradicts (v) of Proposition \ref{pro:11:0:a}. Thus there exists $\zeta_0$ with $C_0(\zeta_0)=0$ where the fact $\zeta_0\neq 0$ follows from (iv).

\subsection{Localization of zeros}
\label{sec:localization-zeros}

Now we know that $ C_{0}(\zeta) $ vanishes somewhere, we would like
to find out where exactly this happens. We are going to begin with a
symmetry result:

\begin{lemma}
  \label{lemma:sym}
The Stokes coefficient $ C_{0}(\zeta) $ verifies the equivalence
\begin{equation}
  \label{eq:9}
  C_{0}(\zeta) = 0 \iff C_{0}(\overline{\omega \zeta})=0.
\end{equation}

\end{lemma}
\noindent
\begin{proof} We see how $ \overline{\mathcal
  Y_{0}(\overline{y};\overline{\zeta})} $ is a solution of
(\ref{eq:5}) whose asymptotic behavior in the sector $ S_{0} $ is the
same as that of $ \mathcal{Y}_{0}(y;\zeta) $. The uniqueness of the
canonical Sibuya solution implies thus that
\[
\overline{\mathcal{Y}_{0}(y;\zeta)} =
\mathcal{Y}_{0}(\overline{y};\overline{\zeta})
\]

Recall that $ {\mathcal Y}_k(y;\zeta)={\mathcal
  Y}(\omega^{-k}y;\omega^{-2k}\zeta) $ and that

\[
{\mathcal Y}_k(y;\zeta)=C_k(\zeta){\mathcal Y}_{
k+1}(y;\zeta)- \omega{\mathcal Y}_{
k+2}(y;\zeta).
\]

It is easy to verify that $ \overline{\mathcal{Y}_{4}(y;\zeta)} =
\mathcal{Y}_{1}(\overline{y};\overline{\zeta}) $ and that $ \overline{\mathcal{Y}_{1}(y;\zeta)} =
\mathcal{Y}_{4}(\overline{y};\overline{\zeta}) $. We conjugate 
\begin{equation}
  \label{eq:19}
  {\mathcal Y}_4(y;\zeta)=C_4(\zeta){\mathcal Y}_{
0}(y;\zeta)- \omega{\mathcal Y}_{
1}(y;\zeta)
\end{equation}
   
 and have

 \begin{equation}
  \label{eq:20}
   {\mathcal Y}_1(\overline{y};\overline{\zeta})= \overline{C_4(\zeta)}{\mathcal Y}_{
0}(\overline{y};\overline{\zeta})- \overline{\omega}{\mathcal Y}_{
4}(\overline{y};\overline{\zeta}).
 \end{equation}
 
Switch to $ \overline{y} $ and $ \overline{\zeta} $ in (\ref{eq:19})
and we get

\begin{equation}
  \label{eq:21}
  {\mathcal Y}_4(\overline{y};\overline{\zeta})=C_4(\overline{\zeta}){\mathcal Y}_{
0}(\overline{y};\overline{\zeta})- \omega{\mathcal Y}_{
1}(\overline{y};\overline{\zeta}).
\end{equation}

Multiplying (\ref{eq:20}) by $ \omega^{-3/4} $ and (\ref{eq:21}) by $
\omega^{3/4} $ we get:

\[
\omega^{-3/4} {\mathcal Y}_1(\overline{y};\overline{\zeta})=\overline{\omega^{3/4}C_4(\zeta)}{\mathcal Y}_{
0}(\overline{y};\overline{\zeta})+  \omega^{3/4}{\mathcal Y}_{
4}(\overline{y};\overline{\zeta}),
\]

\[
\omega^{3/4} {\mathcal Y}_4(\overline{y};\overline{\zeta})=\omega^{3/4}C_4(\overline{\zeta}){\mathcal Y}_{
0}(\overline{y};\overline{\zeta})+  \omega^{-3/4}{\mathcal Y}_{
1}(\overline{y};\overline{\zeta}).
\]

Adding these two equations we have:

\[
\Bigl(\overline{\omega^{3/4}C_4(\zeta)} + \omega^{3/4}C_4(\overline{\zeta})\Bigr){\mathcal Y}_{
0}(\overline{y};\overline{\zeta}) =0,
\]

from which we have

\[
\overline{C_4(\zeta)} = 0 \iff C_4(\overline{\zeta}) =0,
\]

that is

\[
\overline{C_0(\omega^{2}\zeta)} = 0 \iff
C_0(\omega^{2}\overline{\zeta}) =0,
\]

therefore

\[
\overline{C_0(\zeta)} = 0 \iff
C_0(\omega^{4}\overline{\zeta}) = C_0(\overline{\omega\zeta}) =0,
\]
this last equality proving the Lemma.

\end{proof}

The following is a very important step in the construction of the null
solutions, and is the sharpest result, at least to the authors'
knowledge, on the location of the zeros of the entire function $ C_{0}(\zeta) $.

\begin{lemma}
\label{local}
  There exists $ \zeta_{0} \in S = \{z \in\C | \pi < \arg z \leq \frac{19}{15}\pi \} $
  where  $ C_{0}(\zeta_{0}) =0 $.
\end{lemma}

\begin{proof}
  We recall from Proposition 3.1 in \cite{NZ} that $ C_{0}(\zeta) = 0 $
  implies either $ \zeta \in S_{1} = \{\pi \leq \arg\zeta \leq
  \frac{19}{15}\pi \} $ or $ \zeta \in S_{2} = \{\frac{\pi}{3} \leq \arg\zeta \leq
  \frac{3}{5}\pi \}  $. But $ S_{1}  $ and $ S_{2} $ are symmetric
  under the mapping $ \zeta \rightarrow \overline{\omega\zeta} $. We
  just have to show the $ \arg\zeta \neq \pi $. Proposition
  (\ref{pro:11:0:c}) and Lemma (\ref{lemma:sym}) above  together imply that $
  C_{0}(\zeta) \neq 0 $ if $ \zeta $ is real.
\end{proof}

\subsection{Proof of Theorem (\ref{thm:main:1})}
\label{sec:final-proof}


Consider again the operator:

\begin{equation}
  \label{eq:1}
  P_{3}(x,D) = D_{0}^{3} - (D_{1}^{2} + x_{1}^{2}D_{n}^{2})D_{0} - b_{0} x_{1}^{3}D_{n}^{3}.
\end{equation}

In the following we will choose $ b_{0} = \frac{\sqrt{2}}{3\sqrt{3}}
$, which clearly satisfies the hyperbolicity assumption $ b_{0}^{2}
\leq 4/27 $.

Let $ \lambda > 0$ a positive large parameter, $ R > 0,\theta \in ]0,\pi[ $ to
be chosen later and consider:

\begin{equation}
  \label{eq:7}
  U(x,\lambda,R,\theta) = e^{ix_{0}\lambda^{\frac{1}{2}}Re^{i\theta} +
    ix_{n}\lambda}u(x_{1},\lambda,R,\theta)
\end{equation}

Here $ x = (x_{0},x_{1},x'',x_{n}) $, sometimes the $ x'' $ components
will be omitted to enhance readability.

From (\ref{eq:7}) let us set

$$
U(x,\lambda,R,\theta) = E(x_{0},x_{n},\lambda)\times w(A x_{1} + B),
$$

with $ E(x_{0},x_{n},\lambda) =  e^{ix_{0}\lambda^{\frac{1}{2}}Re^{i\theta} +
    ix_{n}\lambda}$ and $ A,B $ to be chosen together with $ w $.

It is clear that $ D_{0}U = \lambda^{1/2}Re^{i\theta}U $, $ D_{n}U =
\lambda U $ and $ D_{1}U = -iEAw'(Ax_{1} +B) $.

Therefore we have

\begin{equation}
\label{eq:71}
PU = U\Bigl(\lambda^{3/2}R^{3}e^{i3\theta} -
\lambda^{5/2}Re^{i\theta}x_{1}^{2} - b_{0}\lambda^{3}x_{1}^{3} +
\lambda^{1/2}Re^{i\theta}A^{2}\frac{w''}{w}(Ax_{1} +B)\Bigr).
\end{equation}

Thus setting $ y = Ax_{1} + B $ we have from (\ref{eq:71}) and the
request that $ PU=0 $,
\begin{equation}
\label{eq:72}
\begin{split}
w''(y) =
\lambda^{-1/2}R^{-1}e^{-i\theta}A^{-2}\Bigl[\frac{b_{0}\lambda^{3}}{A^{3}}y^{3}
+ \Bigl(-3\frac{b_{0}\lambda^{3}B}{A^{3}} +
\frac{\lambda^{5/2}Re^{i\theta}}{A^{2}}\Bigr)y^{2}\\
+ \Bigl(\frac{3b_{0}\lambda^{3}B^{2}}{A^{3}} -
\frac{2\lambda^{5/2}Re^{i\theta}B}{A^{2}}\Bigr)y
-\frac{b_{0}\lambda^{3}B^{3}}{A^{3}} +
\frac{\lambda^{5/2}Re^{i\theta}B^{2}}{A^{2}} - \lambda^{3/2}e^{i3\theta}R^{3}\Bigr]w(y).
\end{split}
\end{equation}

The following choices are then made:

\begin{equation}
\label{eq:73}
\lambda^{-1/2}R^{-1}e^{-i\theta}\frac{b_{0}\lambda^{3}}{A^{5}} = 1,
\end{equation}

and
\begin{equation}
\label{eq:74}
-\frac{3b_{0}\lambda^{3}B}{A^{3}} +
\frac{\lambda^{5/2}Re^{i\theta}}{A^{2}} =0.
\end{equation}

(\ref{eq:73}) and (\ref{eq:74}) yield

\begin{equation}
\label{eq:75}
A = \lambda^{1/2}b_{0}^{1/5}R^{-1/5}e^{-i\theta/5} ,\quad B = \frac{R^{4/5}b_{0}^{-4/5}e^{i4\theta/5}}{3}.
\end{equation}

Using these values we have from (\ref{eq:72})

\begin{equation}
\label{eq:76}
w''(y) = (y^{3} + \zeta y + \mu)w(y),
\end{equation}
with

\begin{equation}
\label{eq:77}
\zeta = - \frac{b_{0}^{-8/5}e^{i8\theta/5}R^{8/5}}{3}, ~~~ \mu =
\lambda R^{2}e^{2i\theta}A^{-2}\Bigl( \frac{2}{27b_{0}^{2}} - 1\Bigr).
\end{equation}

It is now clear that choosing $ b_{0} = \frac{\sqrt{2}}{3\sqrt{3}} $
will give us equation (\ref{eq:5}).

If we do not impose this last condition we would be left with the more
difficult task of finding $ \theta $ and $ R $ such that

$$ 
C_{0}\Bigl( -
\frac{b_{0}^{-8/5}e^{i8\theta/5}R^{8/5}}{3}, R^{12/5}e^{12i\theta/5}b_{0}^{-2/5}(
\frac{2}{27b_{0}^{2}} - 1)\Bigr) =0 .
$$

\medskip

We now choose $ w(y;\zeta) =  {\mathcal Y}_0(y;\zeta_{0}) $ with $ \zeta_{0}
$ found in Lemma \ref{local} and from \eqref{eq:75}  we take $  y =
b_{0}^{\frac{1}{5}}R^{-\frac{1}{5}}\lambda^{\frac{1}{2}}e^{-i\frac{\theta}{5}}x_{1}
+ \frac{1}{3}b_{0}^{-\frac{4}{5}}R^{\frac{4}{5}}e^{\frac{4i\theta}{5}}
$.

We have that $
\frac{1}{3}b_{0}^{-\frac{8}{5}}R^{\frac{8}{5}}e^{i\frac{8\theta}{5}+
  i\pi} = |\zeta_{0}|e^{i\arg \zeta_{0}}  $ and $ \pi < \arg\zeta_{0}
\leq \frac{19}{15}\pi $. 

This clearly leaves us with $ 0 < \theta_{0} = \theta(\arg\zeta_{0}) \leq
\frac{\pi}{6} $, while the number $ R $, still at our disposal, is
chosen to fix the absolute values, thus $ R=R_{0} > 0 $, depending on
$ b_{0} $ and $ |\zeta_{0}| $.

Recall that $ {\mathcal Y}_k(y;\zeta)={\mathcal
  Y}(\omega^{-k}y;\omega^{-2k}\zeta) $ and that

\begin{equation}
\label{STE}
{\mathcal Y}_0(y;\zeta_{0})= - \omega{\mathcal Y}_{
2}(y;\zeta_{0}) = - \omega{\mathcal
  Y_{0}}(\omega^{-2}y;\omega^{-4}\zeta_{0}) ,
\end{equation}

since $ C_{0}(\zeta_{0}) = 0 $.

Thus we notice that when $ x_{1} > 0 $ and $ \lambda $ is large $
\arg(y) \in [-\frac{\pi}{30},0[ $ clearly well inside the subdominant
sector $ S_{0} $.

On the other hand if $ x_{1} < 0 $ and $ \lambda $ is large, using
(\ref{STE}), we have that $ \arg(y) \in [\frac{\pi}{6},\frac{\pi}{5}[
$, again within the subdominant
sector $ S_{0} $.

This proves in particular that $
u(x_{1},\lambda,R_{0},\theta_{0}) $ is, for every $ \lambda > 0
$ in the Schwartz space $ \mathcal{S}(\R) $ and moreover $
u(x_{1},\lambda,R_{0},\theta_{0}) $ is bounded on ${\mathbb R}$ uniformly in $\lambda$. 

Let
\[
U_{\lambda}=e^{i(T-x_0)\lambda^{1/2}Re^{i\theta}-ix_n\lambda}w(y,\zeta_0)
\]
then $PU_{\lambda}=0$ because $P(x_1,-D_0,D_1, -D_n)=-P(x_1,D_0,D_1,D_n)$. Let $u$ be a solution to
the Cauchy  problem
\[
\left\{\begin{array}{ll}
Pu=0,\\
u(0,x')=0,\;D_0u(0,x')=0,\;D_0^2u(0,x')={\bar \phi}(x_1){\bar\psi}(x''){\bar\theta}(x_n)
\end{array}\right.
\]
where $\phi\in C_0^{\infty}(\R)$, $\psi\in C_0^{\infty}(\R^{n-2})$, $\theta\in C_0^{\infty}(\R)$.  Let us set
\[
D_{\delta}=\{x\in {\mathbb R}^{n+1}\mid |x'|^2+|x_0|<\delta\}
\]
and recall the Holmgren theorem (see for example 
\cite{Mi} Theorem 4.2): 
\begin{proposition}
\label{pro:Holmgren}
There exists $\epsilon_0>0$ such that; let $0<\epsilon<\epsilon_0$ and $u(x)\in C^2(D_{\epsilon})$ verifies 
\[
\left\{\begin{array}{l}
Pu=0\quad \mbox{in~} D_{\epsilon}\\
D_0^ju(0,x')=0,\quad j=0,1,\quad x\in D_{\epsilon}\cap\{x_0=0\}
\end{array}\right.
\]
then $u(x)$ vanishes identically in $D_{\epsilon}$.
\end{proposition}

From this proposition we can assume that 
\[
u(x)=0
\]
if $0\leq x_0\leq T$, $|x'|\geq r$ for small $T>0$, $r>0$. Then 
\begin{align*}
0=\int_0^T(PU_{\lambda},u)dx_0=\int_0^T(U_{\lambda},Pu)dx_0-i(D_0^2U_{\lambda}(T),u(T))\\
-i(D_0U_{\lambda}(T), D_0u(T))-i(U_{\lambda}(T),D_0^2(T))+i(U_{\lambda}(0),D_0^2u(0))\\
+i((D_1^2+x_1^2D_n^2)U_{\lambda}(T),u(T)).
\end{align*}
Hence
\begin{align*}
(U_{\lambda}(0), D_0^2u(0))=(D_0^2U_{\lambda}(T),u(T))\\
+(D_0U_{\lambda}(T), D_0u(T))+(U_{\lambda}(T),D_0^2(T))\\
-((D_1^2+x_1^2D_n^2)U_{\lambda}(T),u(T)).
\end{align*}
The right-hand side is $O(\lambda^2)$ because $w(y,\zeta_0)$, $\lambda^{-1/2}D_1w(y,\zeta_0)$ are bou\-nded uniformly in $\lambda$. On the other hand the left-hand side is
\begin{align*}
{\hat\theta}(\lambda)e^{iT\lambda^{1/2}Re^{i\theta}}\int w(y,\zeta_0)\phi(x_1)\psi(x'')dx_1dx''\\
={\hat\theta}(\lambda)e^{iT\lambda^{1/2}R^{1/2}e^{i\theta}}\Big(\psi(x'')dx''\Big)\int w(y,\zeta_0)\phi(x_1)dx_1.
\end{align*}
We choose $\psi$ so that $\int \psi(x'')dx''\neq 0$. Recall that $\theta\in \gamma_0^{(2)}(\R)$ if and only if we have
\[
|{\hat \theta}(\xi)|\leq Ce^{-L|\xi|^{1/2}}
\]
with some $L>0$, $C>0$. Thus if we take $\theta\not\in \gamma_0^{(2)}(\R)$ which is even then $\rho^{-N}{\hat\theta}(\lambda)e^{iT\lambda^{1/2}Re^{i\theta}}$ is not bounded as $\lambda\to\infty$. We must check that
\[
\int w(y,\zeta_0)\phi(x_1)dx_1\to c\neq 0
\]
with a suitable choice of $\phi$. Let $\alpha=b_0^{1/5}R^{-1/5}e^{-i\theta/5}$, $\beta=b_0^{-4/5}R^{4/5}e^{4i\theta/5}/3$. We have
\begin{align*}
\int w(\lambda^{1/2}\alpha x_1+\beta,\zeta_0)\phi(x_1)dx_1=\lambda^{-1/2}\int w(\alpha x_1+\beta,\zeta_0)\phi(\lambda^{-1/2}x_1)dx_1\\
=\lambda^{-1/2}\Big[\sum_{k=0}^2\frac{\lambda^{-k/2}}{k!}\phi^{(k)}(0)\int w(\alpha x_1+\beta,\zeta_0)x_1^kdx_1+O(\lambda^{-3/2})\Big].
\end{align*}
It is enough to show that we have $\int w(\alpha x_1+\beta,\zeta_0)x_1^kdx_1\neq 0$ for at least one $k=0,1,2$. Put
\[
v(\xi)=\int e^{-ix\xi}w(\alpha x+\beta,\zeta_0)dx
\]
Then $v(\xi)$ satisfies the equation
\[
\big(i\alpha \frac{d}{d\xi}+\beta\big)^3v(\xi)+\zeta_0\big(i\alpha \frac{d}{d\xi}+\beta\big)v(\xi)+\alpha^{-2}\xi^2v(\xi)=0
\]
and
\[
v^{(k)}(0)=\int (-ix)^kw(\alpha x+\beta,\zeta_0)dx=(-i)^k\int w(\alpha x+\beta,\zeta_0)x^kdx.
\]
So if $v^{(k)}(0)=0$ for $k=0,1,2$ then we would have $v(\xi)=0$ so that $w(\alpha x+\beta,\zeta_0)=0$ which is a contradiction.

\section{Cones and Factorization}
\label{sec:factorization}

Here we briefly verify that the propagation cone is not transversal to
the triple manifold.

Let $ p(x,\xi) = \xi_{0}^{3} - (\xi_{1}^{2} +
x_{1}^{2}\xi_{n}^{2})\xi_{0} - b_{0}x_{1}^{3}\xi_{n}^{3} $ be the
principal symbol of the operator (\ref{eq:10}). $p$ vanishes exactly
of order $ 3 $ on $ \Sigma_{3} = \{x_{1}=\xi_{0}= \xi_{1} =0 \} $ near $
(0;0,\ldots,1) $ if $ |b_{0}|<\frac{2}{3\sqrt{3}} $. Fix $ z \in
\Sigma_{3} $ and take $ \delta v = (-1,0,\ldots,0;0) $. Clearly $
\delta v  \in T_{z}\Sigma_{3} $ and, since $ \sigma(\delta v, (\delta
y, \delta \eta)) = -\delta \eta_{0} \leq 0 $ if $ (\delta y, \delta
\eta) \in \Gamma_{z} $, we have that $ C_{z} \cap T_{z}\Sigma_{3}
\neq\emptyset $. On the other hand $ C_{z} $ cannot be completely
contained in $  T_{z}\Sigma_{3} $, because otherwise $
T_{z}^{\sigma}\Sigma_{3} \subset C_{z}^{\sigma} $ and this would imply
that $\langle{H_{\xi_{0}},H_{\xi_{1}},H_{x_{1}}}\rangle\subset
\overline{\Gamma_{z}} $, which is false. Therefore $ C_{z} $ is
neither disjoint from nor totally inside $ T_{z}\Sigma_{3} $.

For the next item we change slightly the notations in order to
simplify the treatment of a third degree equation naturally associated
with the problem. Let us show that for our model no root is $ C^{\infty} $.

Let $ p = \tau^{3} - 3(x^{2} + \xi^{2})\tau - 2bx^{3} $, with $ 0 <
|b| <1 $. If $ p $
could be written like
\[
p= ( \tau - L(x,\xi))(\tau^{2} + A(x,\xi) \tau + B(x,\xi))
\]

with $ L,A,B $ regular $ C^{\infty} $ functions, one then would get $ A=L $, $ L^{2} - B = 3(x^{2} + \xi^{2})  $ and $ LB =
2bx^{3}  $.

This shows that at $ x=0 $ there should always exist a regular root $
\tau(0,\xi) = 0 $ identically.

The
discriminant is $ \Delta = 108\{(x^{2} + \xi^{2})^{3} - b^{2}x^{6}\}
$. Putting $ p = -3(x^{2} + \xi^{2}), q = -
2bx^{3}  $ we have

\[
-\frac{q}{2} = \rho\cos\phi, \quad \sqrt{\frac{\Delta}{108}} =
\rho\sin\phi
\]

with 

\[
\rho = \sqrt{\Bigl(-\frac{p}{3}}\Bigr)^{3} = (x^{2} + \xi^{2})^{3/2},
\quad \cos\phi= - \frac{q}{2\rho}=\frac{bx^{3}}{(x^{2} +
  \xi^{2})^{3/2}}.
\]

Thus we have

\[
\phi(x,\xi) = \arccos\Bigl(\frac{bx^{3}}{(x^{2} +
  \xi^{2})^{3/2}}\Bigr).
\]

The root vanishing identically when $ x=0 $ is 

\[
\tau(x,\xi) = 2(x^{2} + \xi^{2})^{1/2}\cos\Bigl(\frac{ \arccos\Bigl(\frac{bx^{3}}{(x^{2} +
  \xi^{2})^{3/2}}\Bigr) + 4\pi}{3}\Bigr).
\]

We have

\[
\arccos\Bigl(\frac{bx^{3}}{(x^{2} +
  \xi^{2})^{3/2}}\Bigr) = \frac{\pi}{2} - f(x,\xi).
\]

with $ f(x,\xi) = g\Bigl(\frac{bx^{3}}{(x^{2} + \xi^{2})^{3/2}}\Bigr)
$, and

\[
g(u) = \sum_{k=0}^{\infty}\frac{(2k)!u^{2k+1}}{2^{2k}(k!)^{2}(2k+1)},
\quad |u| <1.
\]

Therefore we have, since $ |b| < 1 $,

\[
\tau(x,\xi) = 2(x^{2} + \xi^{2})^{1/2}\cos\Bigl(\frac{3\pi}{2} -
\frac{1}{3}g\Bigl(\frac{bx^{3}}{(x^{2} + \xi^{2})^{3/2}}\Bigr)\Bigr)
\]

\[
=
- 2(x^{2} +
\xi^{2})^{1/2}\sin\Bigl(\frac{1}{3}g\Bigl(\frac{bx^{3}}{(x^{2} +
  \xi^{2})^{3/2}}\Bigr)\Bigr).
\]

This implies

\[
\tau(x,\xi) \sim - \frac{2}{3}\frac{bx^{3}}{x^{2} + \xi^{2}}=
-\frac{2}{3}x\rho(x,\xi),
\]

with $ \rho(x,\xi) $, not identically zero because of $ b \neq 0 $
cannot be continuous at the origin: this contradiction
proves that there cannot be a smooth factorization for $ p $.

\end{document}